\documentclass{article}
\usepackage{amsmath}
\usepackage{amscd}
\usepackage{amsthm}
\usepackage{amssymb} \usepackage{latexsym}
\usepackage{eufrak}
\usepackage{euscript}
\usepackage{epsfig}
\usepackage{graphics}
\usepackage{array}
\usepackage{enumerate}

\theoremstyle{theorem}
\newtheorem{theorem}{Theorem}[section]

\theoremstyle{corollary}

\theoremstyle{lemma}
\newtheorem{lemma}{Lemma}[section]
\theoremstyle{definition}

\theoremstyle{proof}

\theoremstyle{remark}

\newcommand{\bel}[1]{\begin{equation}\label{#1}}

\newcommand{\be}{\begin{equation}}

\newcommand{\ba}{\begin{eqnarray}}
\newcommand{\ea}{\end{eqnarray}}
\newcommand{\rf}[1]{(\ref{#1})}
\newcommand{\bi}{\bibitem}
\newcommand{\qe}{\end{equation}}

\begin{document}
\title{On the  spectrum of the normalized graph Laplacian}
\author{Anirban Banerjee, J\"urgen Jost\footnote{Max Planck Institute for Mathematics in the
  Sciences, Inselstr.22, 04103 Leipzig, Germany, banerjee@mis.mpg.de, jost@mis.mpg.de}}

\maketitle

\begin{abstract}
The spectrum of the normalized graph Laplacian yields a very
comprehensive set of invariants of a graph. In order to understand the
information contained in those invariants better, 
  we systematically investigate the behavior of this spectrum under local and global operations like motif doubling,
  graph joining or splitting. The eigenvalue $1$ plays a particular
  role, and we therefore emphasize those constructions that change its
  multiplicity in a controlled manner, like the iterated duplication
  of nodes. 
\end{abstract}

Let $\Gamma$ be a finite and connected graph with $N$ vertices. Two vertices $i,j
\in\Gamma$ are called neighbors, $i\sim j$, when they are connected by
an edge of $\Gamma$. For a vertex $i \in
\Gamma$, let $n_i$ be its degree, that is, the number of its
neighbors. For functions $v$ from the vertices of $\Gamma$ to
$\mathbb{R}$, we define the (normalized) Laplacian as
\bel{3}
\Delta v(i):=  v(i) -\frac{1}{n_i} \sum_{j, j \sim i}v(j) .
\end{equation}
This is different from the operator  $Lv(i):=n_i v(i)-\sum_{j, j \sim i}v(j)$
  usually   studied in the graph theoretical literature as the
  (algebraic) graph Laplacian, see e.g. \cite{Bol,GoRo,Merris,Mohar,BiLeSt}, but
  equivalent to the Laplacian investigated in \cite{Chung}. This
  normalized Laplacian is, for example, the operator underlying random
  walks on graphs, and in contrast to the algebraic Laplacian, it
  naturally incorporates a conservation law. \\
We are interested in the spectrum of this operator as yielding
important invariants of the underlying graph $\Gamma$ and
incorporating its qualitative properties. As in the case of the algebraic
Laplacian, one can essentially recover the graph from its spectrum, up
to isospectral graphs. The latter are known to exist, but are
relatively rare and qualitatively quite similar in most respects (see e.g. 
\cite{ZW} for a systematic discussion). For
a heuristic algorithm for the algebraic Laplacian which can be easily
modified for the normalized Laplacian, see \cite{IpMi}.\\
We now recall some elementary properties, see e.g. \cite{Chung,JJ1}. The normalized Laplacian, henceforth simply called the Laplacian, is
symmetric for the product
\bel{2}
(u,v):=\sum_{i \in V} n_i u(i)v(i)
\end{equation}
for real valued functions $u,v$ on the vertices of $\Gamma$. $\Delta$ is nonnegative in the sense that
$(\Delta u,u) \ge 0$ for all $u$. \\
From these properties, we conclude that the eigenvalues of $\Delta$
are real and nonnegative, where 
  the eigenvalue equation is
\bel{6}
\Delta u - \lambda u =0.
\end{equation}
A nonzero solution $u$ is called an eigenfunction for the eigenvalue
$\lambda$.\\
The smallest eigenvalue is $\lambda_0 =0$, with a constant eigenfunction. Since we
  assume that $\Gamma$ is connected, this eigenvalue is simple, that
  is
\bel{7}
\lambda_k >0
\end{equation}
for $k>0$ where we order the eigenvalues as
$$ \lambda_0=0 < \lambda_1 \le ... \le \lambda_{N-1}.$$
\bel{8}
\lambda_{N-1} \le 2,
\qe
with equality iff the graph is bipartite. The latter is also
equivalent to the fact that whenever $\lambda$ is an eigenvalue, then so is
$2-\lambda$. \\
For a complete graph of $N$ vertices, we have
\bel{9}
\lambda_1 = ... = \lambda_{N-1}=\frac{N}{N-1},
\qe
that is, the eigenvalue $\frac{N}{N-1}$ occurs with multiplicity
$N-1$. Among all graphs with $N$ vertices, this is the largest
possible value for $\lambda_1$ and the smallest possible value for $\lambda_{N-1}$.\\\\
The eigenvalue equation \rf{6} is
\bel{10}
\frac{1}{n_i}\sum_{j \sim i} u(j)=(1-\lambda)u(i) \text{ for all }i.
\qe
In particular, when the eigenfunction $u$ vanishes at $i$, then also
$\sum_{j \sim i} u(j)=0$, and conversely (except for $\lambda=1$). This observation will be
useful for us below.

\section{The eigenvalue 1}
For the eigenvalue $\lambda=1$, \rf{10} becomes simply
\bel{11}
\sum_{j \sim i} u(j) =0 \text{ for all }i,
\qe
that is, the average of the neighboring values vanishes for each
$i$. We call a solution $u$ of \rf{11} balanced. The multiplicity
$m_1$ of
the eigenvalue 1 then equals the number of linearly independent
balanced functions on $\Gamma$.\\
There is an equivalent algebraic formulation: Let $A=(a_{ij})$ be the
adjacency matrix of $\Gamma$; $a_{ij}=1$ if $i$ and $j$ are connected
by an edge and =0 else. \rf{11} then simply means
\bel{12}
Au=\sum_j a_{ij}u(j) =0,
\qe
that is, the vector $u(j)_{j\in \Gamma}$ is in the kernel of the
adjacency matrix. Thus,
\bel{13}
m_1=\text{dim} \ker A.
\qe
We are interested in the question of estimating the multiplicity of
the eigenvalue 1 on a graph. An obvious method for this is to determine
restrictions on corresponding eigenfunctions $f_1$. We shall do that
by graph theoretical considerations, and in this sense, this
constitutes a geometric approach to the algebraic question of
determining or estimating the kernel of a symmetric 0-1 matrix with
vanishing diagonal. \cite{Bev} systematically investigated the effect of the addition of a single vertex on $m_1$. Here, we are also interested in the effect of more global graph operations.\\
We start with the
following simple observation
\begin{lemma}
Let $q$ be a  vertex of degree 1 in $\Gamma$ (such a $q$ is called a
pending vertex). Then any eigenfunction $f_1$ for the eigenvalue 1
vanishes at the unique neighbor of $q$.
\end{lemma}

\section{Motif doubling, graph splitting and joining}
Let $\Sigma$ be a connected subgraph of $\Gamma$ with vertices $p_1,\dots ,p_m$,
containing all of $\Gamma$'s edges between those vertices. We call
such a $\Sigma$ a motif. The situation we have in mind is where $N$,
the number of vertices of $\Gamma$, is large while $m$, the number of
vertices of $\Sigma$, is small. Let 1 be an eigenvalue of
$\Sigma$ with eigenfunction $f^\Sigma_1$. $f^\Sigma_1$ when extended
by 0 outside $\Sigma$ to all of $\Gamma$ need not be an eigenfunction
of $\Gamma$, and 1 need not even be an eigenvalue of $\Gamma$. 
We can, however, enlarge
$\Gamma$ by doubling the motif $\Sigma$ so that the enlarged graph
also possesses the eigenvalue 1, with a localized eigenfunction:
\begin{theorem}
\label{th1}
Let $\Gamma^\Sigma$ be obtained from $\Gamma$ by adding a copy of the
motif $\Sigma$ consisting of the vertices
$q_1,\dots ,q_m$ and the corresponding connections between them, and connecting each $q_\alpha$ with all $p \notin
\Sigma$ that are neighbors of $p_\alpha$. Then $\Gamma^\Sigma$
possesses the eigenvalue 1, with a localized eigenfunction that is
nonzero only at the $p_\alpha$ and the $q_\alpha$.
\end{theorem}
\begin{proof}
A corresponding eigenfunction is obtained as
\bel{20}
f^{\Gamma^\Sigma}_1(p)=\begin{cases} f^\Sigma_1(p_\alpha)
  \text{ if } p=p_\alpha \in \Sigma \\
-f^\Sigma_1(p_\alpha)
  \text{ if } p=q_\alpha\\
0 \text{ else}.
\end{cases}
\qe
\end{proof} 
The theorem also holds for the case where $\Sigma$ is a single vertex 
$p_1$ (even though such a motif does not possess the eigenvalue 1
itself).  Thus, we can always produce the eigenvalue by {\bf vertex 
  doubling}. This is a reformulation of a result of \cite{Ell}.
 \\
Thus, if we wish to produce a high multiplicity for the eigenvalue 1,
we can perform many vertex doublings. We could either duplicate
different vertices, or we could duplicate one vertex repeatedly. In
fact, the repeated doubling of one vertex leaves a characteristic
trace in the number of certain small motifs in the graph. Let $p_1$ be
a vertex and $q_1$ its double. We consider any motif $\Sigma$ consisting of a
certain collection $p, p', p'',\dots $ of neighbors of $p_1$ together
with their connections to both $p_1$ and $q_1$ and possibly some
connections among them. 
\begin{theorem}
Let the graph $\bar{\Gamma}$ be obtained from $\Gamma$ by $n$
successive doublings of the vertex $p_1$, and let $\Sigma$ be any
motif of the type just described. Then $\bar{\Gamma}$ contains at
least
$n \choose 2$ instances of the motif $\Sigma$.
\end{theorem}
\begin{proof}
An instance of the motif $\Sigma$ is obtained by taking any two copies
of $p_1$ and the vertices $p, p', p'',\dots $ together with the
connections defining $\Sigma$. There exist $n \choose 2$ such pairs of
copies of $p_1$ in  $\bar{\Gamma}$.
\end{proof}

Theorem \ref{th1}, however, does not apply to eigenvalues other than 1
because for $\lambda\ne 1$, the vertex degrees $n_i$ in \rf{10} are important,
and this is affected by embedding the motif $\Sigma$ into another
graph $\Gamma$. However, we have the following variant in the general
case. 
\begin{theorem}
Let  $\Sigma$ be a motif in $\Gamma$. Suppose $f$ satisfies
\bel{21}
\frac{1}{n_i}\sum_{j\in \Sigma, j \sim i} f(j)=(1-\lambda)f(i) \text{
  for all }i\in \Sigma \text{ and some } \lambda.
\qe
Then the motif doubling of Theorem \ref{th1} produces a graph
$\Gamma^\Sigma$ with eigenvalue $\lambda$ and an eigenfunction
$f^{\Gamma^\Sigma}$ agreeing with $f$
on $\Sigma$, with $-f$ on the double of $\Sigma$, and being 0 on the
rest of $\Gamma^\Sigma$.
\end{theorem}
\begin{proof}
\rf{21} implies that $f$ satisfies the eigenvalue equation on
$\Sigma$, and therefore $-f$ satisfies it on its double. As before,
the doubling has the effect that for all other vertices $j\in \Gamma^\Sigma$, 
\bel{22}
\frac{1}{n_j}\sum_{\ell \sim j}f^{\Gamma^\Sigma}(\ell)=0.
\qe
\end{proof}
The simplest motif is an edge connecting two vertices $p_1,p_2$. The
corresponding relations \rf{21} then are
\bel{23}
\frac{1}{n_{p_1}}f(p_2)=(1-\lambda)f(p_1),\quad
\frac{1}{n_{p_2}}f(p_1)=(1-\lambda)f(p_2)
\qe
which admit the solutions 
\bel{24}
\lambda=1\pm \frac{1}{\sqrt{n_{p_1}n_{p_2}}}.
\qe
Thus edge doubling leads to those eigenvalues which when $p_1$ or
$p_2$ has a large degree become close to 1. In any case, the two
values are symmetric about 1.\\\\
We can also double the entire graph:
\begin{theorem}
Let $\Gamma_1$ and $\Gamma_2$ be isomorphic graphs with vertices $p_1,
\dots ,p_n$ and $q_1, \dots ,q_n$ respectively, where $p_i$
corresponds to $q_i$, for $i = 1, \dots , n$. We then construct a
graph $\Gamma_0$ by connecting $p_i$ with $q_j$ whenever $p_j \sim p_i$.
If $\lambda_1, \dots , \lambda_n$ are the eigenvalues of  $\Gamma_1$
and $\Gamma_2$, then $\Gamma_0$ has these same eigenvalues, 
 and the eigenvalue 1 with
multiplicity $n$.
\end{theorem}
\begin{proof}
Since the degree of every vertex $p$ in $\Gamma_0$ is $2n_p$ where $n_p$
is its original degree in $\Gamma_1$, we have for an
eigenfunction $f_\lambda$ of $\Gamma_1$ (which then is also an
eigenfunction on $\Gamma_2$),
\bel{24a}
\frac{1}{2n_p}\sum_{s\in \Gamma_0,s \sim p}
f_\lambda(s)=\frac{1}{n_p}\sum_{s\in \Gamma_1,s \sim p} f_\lambda(s) =(1-\lambda)f_\lambda(p).
\qe
Thus, by \rf{10}, it is an eigenfunction on $\Gamma_0$.\\
Finally, similarly to the proof of Theorem \ref{th1}, we obtain the eigenvalue 1 with
multiplicity $n$: for each $p \in \Gamma_1$, we construct an
eigenfunction with value 1 at $p$, $-1$ at its double in $\Gamma_2$,
and 0 elsewhere.
\end{proof}

We now turn to a different operation. 
Let  $\Gamma$ be a graph with an eigenfunction $f_1$. We
arbitrarily divide $\Gamma$ into subgraphs
$\Sigma_0,\Sigma_1,\Sigma_2$ such that there is no edge between an
element of $\Sigma_1$ and an element of $\Sigma_2$. We then take the
graphs $\Gamma_1=\Sigma_1 \cup \Sigma_0$ and  $\Gamma_2=\Sigma_2 \cup
\Sigma_0$, in such a manner that each edge between two elements of
$\Sigma_0$ is contained in either $\Gamma_1$ or $\Gamma_2$, but not in
both of them, and form a connected graph $\Gamma_0$ by taking an additional vertices
$w$ for each vertex $q\in \Sigma_0$ and connect it with  the two copies of $q$
in  $\Gamma_1$ and $\Gamma_2$.
\begin{theorem}
$\Gamma_0$ possesses the eigenvalue 1 with an eigenfunction that
agrees with $f_1$ on $\Gamma_1$.
\end{theorem}
\begin{proof}
We put
\bel{25}
f^{\Gamma_0}_1(p)=\begin{cases} f_1(p) \text{ for } p \in \Gamma_1
  \\
-f_1(p) \text{ for } p \in  \Gamma_2 \\
-\sum_{s \in \Gamma_1, s\sim q}f_1(s) \text{ when }p=w  \text{ is one of
  the added vertices connected to } q\in \Gamma_1
\end{cases}
\qe
This works out  because $\sum_{s \in \Gamma_1, s\sim q}f_1(s)+\sum_{s
  \in \Gamma_2, s\sim q}f_1(s)=\sum_{s \in \Gamma, s \sim q}f_1(s)=0$
since $f_1$ is an eigenfunction on $\Gamma$.
\end{proof}

A simple and special case consists in taking a node $p$ and joining a
chain of length 2 to it, that is, connect $p$ with a new node $p_1$
and that node in turn with another new node $p_2$ and put the value 0 at
$p_1$ and the value $-f_1(p)$ at $p_2$. This case was obtained in \cite{Bev}.\\\\ 
The next operation, graph joining, works for any eigenvalue, not just
1:
\begin{theorem}
\label{th5} 
Let $\Gamma_1, \Gamma_2$ be graphs with the same eigenvalue
$\lambda$ and corresponding eigenfunctions $f^1_\lambda,
f^2_\lambda$. Assume that $f^1_\lambda(p_1)=0$ and
$f^2_\lambda(p_2)=0$ for some $p_1 \in \Gamma_1,p_2 \in
\Gamma_2$. Then the graph $\Gamma$ obtained by joining $\Gamma_1$ and
$\Gamma_2$ via identifying $p_1$ with $p_2$ also has the eigenvalue
$\lambda$ with an eigenfunction given by $f^1_\lambda$ on $\Gamma_1$,
$f^2_\lambda$ on $\Gamma_2$.
\end{theorem}
\begin{proof}
We observe from \rf{10} that for an  eigenfunction $f_\lambda$ whenever
$f_\lambda(q)=0$ at some $q$, then also $\sum_{s\sim
  q}f_\lambda(s)=0$. This applies to $p_1$ and $p_2$, and therefore,
we can also join the eigenfunctions on the two components.
\end{proof}

This includes the case where  either $f^1_\lambda$ or  $f^2_\lambda$ is
identically 0. \\
{\bf Example:} A triangle, that is, a complete graph of 3 vertices
$i_1,i_2,i_3$, possesses the eigenvalue 3/2 with multiplicity 2. An
eigenfunction $f_{3/2}$ vanishes at one of the vertices, say $f_{3/2}(i_1)=0$
and takes the values +1 and -1, resp., at the two other ones. Thus,
when a triangle is joined at one vertex to another graph, the
eigenvalue 3/2 is kept. For instance (see \cite{Chung}), the petal graph, that is, a
graph where $m$ triangles are joined at a single vertex, has the
eigenvalue 3/2 with multiplicity $m+1$ (here, $m$ of these eigenvalues
are obtained via the described construction, and the remaining
eigenfunction has the value $-2$ at the central vertex where all the
triangles are joined and 1 at all other ones).\\

Also, when the condition of Theorem \ref{th5} is satisfied at several pairs of vertices,
we can form more bonds by vertex identifications between the two graphs.\\
For the eigenvalue 1, the situation is even better: We need not
require  $f^1_\lambda(p_1)=0$ and $f^2_\lambda(p_2)=0$, but only
$f^1_\lambda(p_1)=f^2_\lambda(p_2)$ to make the joining construction
work. 
 \\

\section{Examples}
A chain of $m$ vertices (that is, where we have an edge between
$p_j$ and $p_{j+1}$ for $j=1,\dots ,m-1$), by the lemma and node
doubling,  possesses the eigenvalue 1 (with
multiplicity 1) iff $m$ is odd, with eigenfunction
$f_1(p_1)=1,f_1(p_2)=0,f_1(p_3)=-1,f_1(p_4)=0,\dots $. Similarly, a
closed chain (that is, where we add an edge between $p_m$ and $p_1$) 
possesses the eigenvalue 1 (with multiplicity 2) iff $m$ is a
multiple of 4.\\
Local operations like adding an edge may increase or  decrease $m_1$
or leave it invariant. Adding a pending vertex to a chain of length 2
increases $m_1$ from 0 to 1, adding a pending vertex to closed chain
of length 3, a triangle, leaves $m_1=0$, adding a pending vertex to a
closed chain of length 4, a quadrangle, reduces $m_1$ from 2 to
1 (see \cite{Bev} for general results in this direction). Similarly, closing a chain by adding an edge between the first and
last vertex may increase, decrease or leave $m_1$ the same.\\
In any case, the question of the eigenvalue 1 is not a local one. Take  closed chains of lengths
$4k-1$ and $4\ell +1$. Neither of them supports the eigenvalue 1, but
if we join them at a single point (that is, we take a point $p_0$ in
the first and a point $q_0$ in the second graph and form a new graph
by identifying $p_0$ and $q_0$), the resulting graph has 1 as an
eigenvalue. An eigenfunction has the value 1 at the joined node, and
the values $\pm 1$ occurring always in neighboring pairs in the rest
of the chains, where the two neighbors of $p_0$ in the first chain
both get the value $-1$, and the ones in the second chain the value
1.\\

\section{Construction of graphs with eigenvalue 1 from given data}
Let $f$ be an integer valued function on the vertices of the graph
$\Gamma$. We define the excess of $p\in \Gamma$ as 
\bel{40}
e(p):=\sum_{q \sim p}f(q).
\qe
Thus, $f$ is an eigenfunction for the eigenvalue iff $e(p)=0$ for all
$p$.\\
We are going to show that we can construct graphs $\Gamma$ and functions $f$ with
the property that $e(p)=0$ except for one single vertex $p_0$ where
the pair $(f(p),e(p))$ assumes any prescribed integer values
$(n,m)$. These will be assembled from elementary building blocks. 
\begin{enumerate}
\item A triangle with a function $f$ that takes the value $-1$ at two
  vertices and the value 1 at the third vertex, our $p_0$, realizes
  the pair $(1,-2)$.
\item The same triangle, with a pending vertex, our new $p_0$, connected to the vertex
  with value 1, and given the value 2, realizes $(2,1)$.
\item Joining instead $\ell$ triangles at a single vertex, our $p_0$, with value 1,
  assigning $-1$ to all the other vertices as before, yields $(1,-2\ell)$.
\item A pentagon, i.e., a closed chain of 5 vertices, with value $-1$ at
  two adjacent vertices and 1 at the remaining three, the middle one
  of which is our $p_0$, realizes $(1,2)$.
\item Similarly, adding a pending vertex, again our new $p_0$,
  connected to the former $p_0$ in the pentagon, and assigned the
  value $-2$, realizes $(-2,1)$.
\item Likewise, joining $\ell$ such pentagons instead at $p_0$ yields
  $(1,2\ell)$.
\item In general, connecting a pending vertex as the new $p_0$ to the
  former $p_0$ changes $(n,m)$ to $(-m,n)$.
\item In general, joining the $p_0$s from graphs with values $(n,m_1),\dots
  (n,m_k)$ yields $(n,\sum_1^k m_j)$.
\end{enumerate}
Thus, from the triangle and the pentagon, by adding pending vertices
and graph joining, we can indeed realize all integer pairs $(n,m)$.
\begin{theorem}
Let $\Sigma$ be a graph, $f$ an integer valued function on its
vertices. We can then construct a graph $\Gamma$ containing the motif
$\Sigma$ with  eigenvalue 1 and an eigenfunction coinciding with $f$
on $\Sigma$.
\end{theorem}
\begin{proof}
At each $p \in \Sigma$, we attach a graph realizing the pair
$(f(p),-e(p))$. This ensures \rf{10} at $p$.
\end{proof}

The preceding constructions also tell us how $m_1$, the multiplicity
of the eigenvalue 1, behaves when we modify a graph $\Gamma'$,
consisting possibly  of two disjoint components $\Gamma_1,
\Gamma_2$, by either identifying vertices or by joining vertices by
new edges. The graph resulting from these operations will be called
$\Gamma$. We consider two cases:
\begin{enumerate}
\item We identify the vertex $p_j$  with 
  $q_j$ for $ j=1,\dots ,m$, assuming that they do not have common neighbors. Then
  \begin{enumerate}
\item We can generate  an eigenfunction on $\Gamma$ whenever we find a
  function $g$ on $\Gamma'$  with vanishing excess except
  possibly at the joined points where we require
\bel{61}
g(p_j)=g(q_j)\text{ and }e_{g}(p_j)=-e_{g}(q_j)\text{ for }
j=1,\dots ,m.
\qe
  \item As a special case of \rf{61}, an eigenfunction $f^{\Gamma'}_1$ 
    produces an eigenfunction $f^\Gamma_1$ whenever 
\bel{60} f^{\Gamma'}_1(p_j)=f^{\Gamma'}_1(q_j) \text{ for }
j=1,\dots ,m.
\qe
In the case where $\Gamma'$ consists of two disjoint components
$\Gamma_1,\Gamma_2$, this includes the case where that value is 0 for all $j$ and 
$f^{\Gamma'}_1$ vanishes identically on one of the components. In other words, we can extend an
eigenfunction from $\Gamma_1$, say, to the rest of the graph by 0
whenever that function vanishes at all joining points. \\
Since in general, \rf{60} cannot be satisfied for a basis of
eigenfunctions, by this process, we can only
expect to generate fewer than $m^{\Gamma'}_1$ linearly
independent eigenfunctions on $\Gamma$. 

  \end{enumerate}
Whether $m^\Gamma_1$ is larger or smaller than $m^{\Gamma'}_1$ then depends on the balance between these two
processes, that is, how many eigenfunctions satisfy \rf{60} vs. how
many new eigenfunctions can be produced by functions satisfying
\rf{61} with nonvanishing excess at some of the joined vertices.
\item We connect the vertex $p_j$  by an edge with 
  $q_j$  for $ j=1,\dots ,m$. Then
  \begin{enumerate}
\item We can generate  eigenfunctions on $\Gamma$ whenever we find a 
  function $g$ on $\Gamma'$ with vanishing excess except
  possibly at the connected points where we require
\bel{63}
g(p_j)=-e_{g}(q_j)\text{ and } g(q_j)=-e_{g}(p_j)\text{ for }
j=1,\dots ,m.
\qe
 \item Again, as a special case of \rf{63}, an eigenfunctions $f^{\Gamma'}_1$ 
    produces an eigenfunction $f^\Gamma_1$ whenever 
\bel{62} f^{\Gamma'}_1(p_j)=0=f^{\Gamma'}_1(q_j) \text{ for }
j=1,\dots ,m.
\qe
This imposes a stronger constraint than in \rf{60} on eigenfunctions
to yield an eigenfunction on $\Gamma$.

\end{enumerate}

\end{enumerate}


\begin{thebibliography}{999}
\bi{Bev} J.Bevis, K.Blount, G.Davis, G.Domke, V.Miller, The rank of a graph after vertex addition, Lin.Alg.Appl.265, 1997, 55--69

\bi{BiLeSt} T.B{\i}y{\i}ko{\u{g}}lu, J.Leydold, P.Stadler, Laplacian eigenvectors
of graphs, Springer LNM, to appear

\bi{Bol} B.Bolob\'as, Modern graph theory, Springer, 1998

\bi{Chung} F.Chung, Spectral graph theory, AMS, 1997



\bibitem{DLS}
G.Gladwell, E.Davies, J.Leydold, and P.Stadler,
 Discrete nodal domain theorems,
 Lin.Alg.Appl.336, 2001, 51-60

\bi{Ell} M.Ellingham, Basic subgraphs and graph spectra, Australas.J.Combin.8, 1993, 247--265

\bi{GoRo} C.Godsil, G.Royle, Algebraic graph theory, Springer, 2001

\bi{IpMi} M. Ipsen, A. S. Mikhailov,
  Evolutionary reconstruction of networks,
  Phys. Rev. E 66(4),
  2002

\bi{JJ1} J. Jost, M. P. Joy,
  Spectral properties and synchronization in coupled map lattices,
  Phys.Rev.E 65(1), 2002
 
\bi{Merris}
  R. Merris, Laplacian matrices of graphs -- a survey,
   Lin. Alg.  Appl.198, 1994, 143-176

\bi{Mohar} B. Mohar,
 Some applications of Laplace eigenvalues of graphs, in: G.Hahn,
 G.Sabidussi (eds.),
 Graph symmetry: Algebraic methods and applications, pp. 227-277, Springer, 1997
 
\bi{ZW} P.Zhu, R.Wilson, A study of graph spectra for comparing graphs  
\end{thebibliography}
\end{document}